\documentclass[12pt]{article}
\title{Computing the Chern-Schwartz-MacPherson Class of Complete Simplical Toric Varieties}
\author{
        Martin Helmer  \\ \normalsize
       Department of Mathematics,\\ \normalsize
 University of California, Berkeley,\\ \normalsize
  Berkeley, California,USA \\ \normalsize
94720-3840 \\ \normalsize
  \texttt{martin.helmer@berkeley.edu}
  }
\date{\today}

\usepackage{chngcntr}
\counterwithin{table}{section}
\usepackage{booktabs}% http://ctan.org/pkg/booktabs
\usepackage{colortbl}% http://ctan.org/pkg/colortbl
\usepackage{xcolor}% http://ctan.org/pkg/xcolor
\usepackage{graphicx}% http://ctan.org/pkg/graphicx

 \definecolor{Ftitle}{RGB}{11,46,108}
\definecolor{line}{RGB}{87,39,117}
\colorlet{tableheadcolor}{Ftitle!25} % Table header colour = 25% gray
\colorlet{tablerowcolor}{gray!10} % Table row separator colour = 10% gray
\usepackage[hidelinks]{hyperref}

\newcommand{\CC}{\mathbb{C}}

\newcommand{\ZZ}{ \mathbb{Z}}
\newcommand{\QQ}{ \mathbb{Q}}
\newcommand{\pp}{\mathbb{P}}
\usepackage{xcolor}
\definecolor{mypurple}{HTML}{5B1280}

\usepackage{amsmath,amsthm}

\usepackage{amsfonts}
\usepackage{enumerate}
\newtheorem{theorem}{Theorem}[section]
\newtheorem{propn}[theorem]{Proposition}

\newtheorem{lemma}[theorem]{Lemma}

\newtheorem{defn}[theorem]{Definition}
\newtheorem{algorithm}{Algorithm}
\newtheorem{example}[theorem]{Example}

\newcommand{\RR}{ \mathbb{R}}

\begin{document}
\maketitle
\begin{abstract}
Topological invariants such as characteristic classes are an important tool to aid in understanding and categorizing the structure and properties of algebraic varieties. In this note we consider the problem of computing a particular characteristic class, the Chern-Schwartz-MacPherson class, of a complete simplicial toric variety $X_{\Sigma}$ defined by a fan ${\Sigma}$ from the combinatorial data contained in the fan $\Sigma$. Specifically, we give an effective combinatorial algorithm to compute the Chern-Schwartz-MacPherson class of $X_{\Sigma}$, in the Chow ring (or rational Chow ring) of $X_{\Sigma}$. This method is formulated by combining, and when necessary modifying, several known results from the literature and is implemented in Macaulay2 for test purposes.
\end{abstract}

\section{Introduction and Background}
The Chern-Schwartz-MacPherson ($c_{SM}$) class is a generalization of the total Chern class, that is the Chern class of the tangent bundle, to singular varieties. Unlike other generalizations of the Chern class to the singular setting the $c_{SM}$ class maintains many of the functorial properties of the total Chern class, and in particular maintains the relation to the Euler characteristic. This means, explicitly, that as with the Chern class the $c_{SM}$ class contains the Euler characteristic as the degree of its zero dimensional component, this relationship is discussed in more detail in \S\ref{subsection:Problem}. 

Historically the existence of a functorial theory of Chern classes for singular varieties in terms of a natural transformation from
the functor of constructible functions to some nice homology theory, and its relation to the Euler characteristic, was conjectured by Deligne and Grothendieck in the 1960's. In the 1974 article~\cite{macpherson1974chern}, MacPherson proved the existence of such a transformation, introducing a new notion of Chern classes for singular algebraic varieties. Independently in the 1960's Schwartz \cite{schwartz1965classes} defined a theory of Chern classes for singular varieties in relative cohomology. It was later shown in a paper of Brasselet and Schwartz \cite{brasselet1981classes} that these two different notions were in fact equivalent. This construction is now commonly referred to as the Chern-Schwartz-MacPherson class. 

In this note we present Algorithm \ref{algorithm:toricAlg} which computes the Chern-Schwartz-MacPherson class and/or the Euler characteristic of a complete simplicial toric variety $X_{\Sigma}$ defined by a fan $\Sigma$. The algorithm is based on a result of Barthel, Brasselet and Fieseler \cite{barthel1992classes} which gives an expression for the $c_{SM}$ class of a toric variety in terms of torus orbit closures. Note that, for simplicity, we will only consider toric varieties $X_{\Sigma}$ over $\CC$.

From a computational point of view the problem of calculating the $c_{SM}$ class for subschemes $V$ of $\pp^n$ has been considered by Aluffi in \cite{aluffi2003computing}, by Jost in \cite{jost2013algorithm} and by the author of this note in \cite{helmer2015,helmerTCS}; this problem has also been considered for subschemes of some smooth complete toric varieties by the author in \cite{helmer2015algorithm}. All of these algorithms have at their core the need to solve polynomial systems of varying difficulty; for example by means of Gr\"obner bases calculations or polynomial homotopy continuation. As such the running times of all such algorithms are dependent on the algebraic degrees of the defining equations of $V$ and on other algebraic properties of the defining equations. Given the often substantial computational cost of solving polynomial systems we believe that an approach to computing $c_{SM}$ classes which is strictly combinatorial in nature is desirable in settings where this is possible, such as the toric setting considered here. 

We note that the restriction to complete simplicial toric varieties is not required in the statement of the result of Barthel, Brasselet and Fieseler \cite{barthel1992classes} on which our algorithm is based, indeed these restrictions are present on the algorithm only for the purpose of simplifying the construction of the Chow ring of the toric variety. If one was able to construct the Chow ring in a simple manner with the restrictions removed the algorithm could be applied unchanged in this more general setting.   

The Macaulay2 \cite{M2} implementation of our algorithm for computing the $c_{SM}$ class and Euler characteristic of a complete simplicial toric variety presented in this note can be found at \url{https://github.com/Martin-Helmer/char-class-calc}. This implementation is accessed via the ``CharToric" package. Note that this implementation is also available in the github version of the ``CharacteristicClasses" Macaulay2 package, see \newline \scriptsize \url{https://github.com/Macaulay2/M2/blob/master/M2/Macaulay2/packages/CharacteristicClasses.m2},\normalsize \newline and will be included in the next release of Macaulay2.  %, see Example \ref{example:toric_Csm} and Appendix \ref{app:section:ChToricCode_description} for the package syntax.

\begin{example}
Let $\mathcal{H}_r$ denote the $r$-th Hirzebruch surface (see, for example, Cox, Little, Schenck \cite[Example 3.1.16]{david2011toric}). Taking $r=5$ and letting $R=\CC[x_0,x_1,x_2,x_3,x_4]$ be the total coordinate ring of the toric variety $\mathcal{H}_r$ we have that \begin{equation}
c_{SM}(\mathcal{H}_r)=4x_1x_2+2x_1+7x_2+1 \in A^*(\mathcal{H}_r),
\end{equation}where $A^*(\mathcal{H}_r)$ is the Chow ring of $\mathcal{H}_r$. We may write this as \begin{equation}
A^*(\mathcal{H}_r) \cong \ZZ[x_0,x_1,x_2,x_3,x_4]/(x_0x_2,x_1x_3,x_0-x_2,-x_3+x_1+5x_2).
\end{equation}From this we deduce that the Euler characteristic is $$
\chi(\mathcal{H}_r)=\int c_{SM}(\mathcal{H}_r)=4 ,
$$ where $\int \alpha$ denotes the degree of the zero dimensional part of the cycle class $\alpha$ in some Chow ring (which is the coefficient of $x_1x_2$ in this case). Note that the Euler characteristic could also be obtained directly as the number of $2$-dimensional cones in the fan corresponding to the toric variety $\mathcal{H}_r$ via Theorem 12.3.9 of Cox, Little and Shanek \cite{david2011toric}. 
\end{example}

The content of this note will be organized as follows. In \S\ref{section:ChowRingToric} we will establish the setting for this work and review the construction of the rational Chow ring of a complete and simplicial toric variety. In \S\ref{subsection:Problem} we will state the problem and briefly review the definition of the $c_{SM}$ class. We then review relevant related results in \S\ref{section:csmToric}. In \S\ref{section:AlgNPreform} we detail the construction of our algorithm for computing the $c_{SM}$ class in the setting considered here. The problem of computing the multiplicity of a cone in an explicit manner is considered in \S\ref{subsection:Multiplcites}. Our algorithm for computing $c_{SM}$ classes (Algorithm \ref{algorithm:toricAlg}), along with the results of some performance testing of Algorithm \ref{algorithm:toricAlg} is given in \S\ref{subsection:Algorithm}.

\subsection{Setting and Notation} \label{section:ChowRingToric}
Let $X_{\Sigma}$ be an $n$-dimensional complete and simplicial toric variety defined by a fan $\Sigma$. Similar to the construction of the Chow ring in the smooth case we may construct the Chow ring of $X_{\Sigma}$ from the Chow groups, that is the groups $A^j(X_{\Sigma})$ of codimension $j$-cycles on $X_{\Sigma}$ modulo rational equivalence. The only difference in this case will be that we work over the rational number field $\mathbb{Q}$ rather than the integers. 

Using the definition of the intersection product on rational cycles (see \S{12.5} of \cite{david2011toric}) we have that the rational Chow ring of $X_{\Sigma} $ is given by the graded ring \begin{equation}
A^*(X_{\Sigma})_{\QQ}=A^*(X_{\Sigma}) \otimes_{\ZZ} \QQ =\bigoplus_{j=0}^n A^j(X_{\Sigma}) \otimes_{\ZZ} \QQ. 
\end{equation}

For each cone $\sigma$ in the fan $\Sigma$ the orbit closure $V(\sigma)$ is a subvariety of codimension $\dim(\sigma)$. We will write $[V(\sigma)]$ for the rational equivalence class of $V(\sigma)$ in $A^{\dim(\sigma)}(X_{\Sigma})$.

For convenience of notation we will also write $A_{\ell}(X_{\Sigma})$ for the dimension $\ell$-cycles on $X_{\Sigma}$ modulo rational equivalence. For a more in depth discussion of rational equivalence, Chow groups, and Chow rings see Fulton \cite{fulton}. 

\begin{propn}[Lemma 12.5.1 of \cite{david2011toric}] The collections $[V(\sigma)] \in A_j(X_{\Sigma})$ for $\sigma\in \Sigma$ having dimension $n-j$ generate $A_j(X_{\Sigma})$, the Chow group of dimension $j$. Further the collection $[V(\sigma)]$ for all $\sigma\in \Sigma$ generates $A^*(X_{\Sigma})$ as an abelian group.\label{propn:ChowGroupsBasisToric}
\end{propn}

The following proposition gives us a simple method to compute the rational Chow ring of a complete, simplicial toric variety $X_{\Sigma}$. We will use this result to compute the rational Chow ring $A^*(X_{\Sigma})_{\QQ}$ in Algorithm \ref{algorithm:toricAlg}, our algorithm to compute the $c_{SM}$ class of a complete, simplicial toric variety.
\begin{propn}[Theorem 12.5.3 of Cox, Little, Schenck \cite{david2011toric}] Let $N$ be an integer lattice with dual $M$. Let $X_{\Sigma}$ be a complete and simplicial toric variety with generating rays $\Sigma(1)=\rho_1,\dots, \rho_r $ where $\rho_j=\left\langle v_j \right\rangle$ for $v_j\in N$. Then we have that \begin{equation}
\mathbb{Q}[x_1,\dots,x_r]/(\mathcal{I}+\mathcal{J}) \cong A^*(X_{\Sigma})_{\mathbb{Q}}, 
\end{equation} with the isomorphism map specified by $[x_i] \mapsto [V({\rho_i})]$. Here $\mathcal{I}$ denotes the Stanley-Reisner ideal of the fan $\Sigma$, that is the ideal in $\mathbb{Q}[x_1,\dots,x_r]$ specified by \begin{equation}
\mathcal{I}=(x_{i_1}\cdots x_{i_s} \; | \; i_{i_j} \; \mathrm{distinct} \; \mathrm{and \;} \rho_{i_1} + \cdots + \rho_{i_s} \; \mathrm{is \; not \; a \; cone \; of \;} \Sigma ) 
\end{equation} and $\mathcal{J}$ denotes the ideal of $\mathbb{Q}[x_1,\dots,x_r]$ generated by linear relations of the rays, that is $\mathcal{J}$ is generated by linear forms \begin{equation}
\sum_{j=1}^r m(v_j) x_j
\end{equation} for $m$ ranging over some basis of $M$.
\label{propn:ChowRingDef}
\end{propn}

\subsection{Problem}\label{subsection:Problem}
The main problem considered in this note is the following: given a complete simplical toric variety $X_{\Sigma}$ how do does one efficiently compute the class $c_{SM}(X_{\Sigma})$ in the Chow ring $A^*(X_{\Sigma})_{\QQ}$? We will give a method to solve this problem in Algorithm \ref{algorithm:toricAlg}. To further establish the context for this problem, however, we will briefly discuss the definition of the Chern-Schwartz-MacPherson class.

The total Chern class of a $j$-dimensional nonsingular variety $V$ is defined as the Chern class of the tangent bundle $T_V$, we write this as $c(V)=c(T_V) \cdot [V]$ in the Chow ring of $V$, $A_*(V)$. See Fulton \cite[\S3.2]{fulton} for a definition of the Chern class of a vector bundle. As a consequence of the Gauss-Bonnet-Chern theorem (or the Grothendieck-Riemann-Roch theorem, see for example Sch\"urmann and Yokura \cite{SchYokura2005}), we have that the degree of the zero dimensional component of the total Chern class of a projective variety is equal to the Euler characteristic, that is \begin{equation}
\int c(T_V) \cdot [V]=\chi(V). \label{eq:chern_euler_non_singular}
\end{equation}Here $\int \alpha$ denotes the degree of the zero dimensional component of the class $\alpha \in A_*(V)$, i.e. the degree of the part of $\alpha$ in $A_0(V)$.

There are several known generalizations of the total Chern class to singular varieties. All of these notions agree with $c(T_V) \cdot [V]$ for nonsingular $V$, however the Chern-Schwartz-MacPherson class is the only one of these that satisfies a property analogous to (\ref{eq:chern_euler_non_singular}) for any $V$, i.e. \begin{equation}
\int c_{SM}(V)=\chi(V). \label{eq:csm_euler}
\end{equation} %As we will see below this relation follows from the functorial nature of the $c_{SM}$ class. 

We review here the construction of the $c_{SM}$ classes, given in the manner considered by MacPherson \cite{macpherson1974chern}. For a scheme $V$, let $\mathcal{C}(V )$ denote the abelian group of finite linear combinations $\sum_W m_W \mathbf{1}_W$, where $W$ are (closed) subvarieties of $V$, $m_W \in \ZZ$, and $\mathbf{1}_W$ denotes the function that is $1$ in $W$, and $0$ outside of $W$. Elements  $f\in \mathcal{C}(V )$ are known as constructible functions and the group  $\mathcal{C}(V )$ is referred to as the group of constructible functions on $V$. To make $\mathcal{C}$ into a functor we let $\mathcal{C}$ map a scheme $V$ to the group of constructible functions on $V$ and a proper morphism $f: V_1 \to V_2$  is mapped by $\mathcal{C}$ to $$\mathcal{C}(f)(\mathbf{1}_W)(p)=\chi(f^{-1}(p) \cap W), \;\;\; W \subset V_1, \; p\in V_2 \; \mathrm{a \; closed \; point}.$$ 

Another functor from algebraic varieties to abelian groups is the Chow group functor $\mathcal{A}_*$. The $c_{SM}$ class may be realized as a natural transformation between these two functors.   
\begin{defn}
The Chern-Schwartz-MacPherson class is the unique natural transformation between the constructible function functor and the Chow group functor, that is $c_{SM}: \mathcal{C}\to \mathcal{A}_*$ is the unique natural transformation satisfying: \begin{itemize}
\item (\textit{Normalization}) $ c_{SM}(\mathbf{1}_V)=c(T_V) \cdot [V] $ for $V $ non-singular  and complete.
\item (\textit{Naturality}) $f_{*}(c_{SM}(\phi))=c_{SM}(\mathcal{C}(f)(\phi))$, for $f:X \to Y$ a proper transformation of projective varieties, $\phi$ a constructible function on $X$. \label{defn:csm_natural_transform}
\end{itemize}
\end{defn}
For a scheme $V$ let $V_{red}$ denote the support of $V$, the notation $c_{SM}(V)$ is taken to mean $c_{SM}(\mathbf{1}_V)$ and hence, since $\mathbf{1}_V=\mathbf{1}_{V_{red}} $, we denote $c_{SM}(V)=c_{SM}(V_{red})$. 

Note that the $c_{SM}$ classes (and constructible functions) also satisfy the same inclusion/exclusion relation as the Euler characteristic, i.e.\ for $V_1,V_2$ subschemes of a scheme $W$ we have
$$
c_{SM}(V_1 \cup V_2) = c_{SM}(V_1) +c_{SM}(V_2) -c_{SM}(V_1\cap V_2).
$$ 

We note that in some settings, such as subschemes of projective spaces or subschemes of some toric varieties, computing the $c_{SM}$ class seems to provide a quite effective means, relative to other available techniques, to compute the Euler characteristic. For a discussion of this see, for example, \cite{helmer2015,helmer2015algorithm}. For toric varieties themselves, however, this is not the case as there is in fact an explicit formula for the Euler characteristic of a toric variety, see Theorem 12.3.9 of Cox, Little and Shanek \cite{david2011toric}. %Specifically for $X_{\Sigma}$ a toric variety of dimension $n$ Thoerem 12.3.9 of Cox, Little and Shanek \cite{david2011toric} states that \begin{equation}
%\chi(X_{\Sigma})=|\Sigma(n)|.
%\end{equation}%We note that this formula, while being direct, is also a combinatorial formula in the sense that one must count all the dimension $n$ cones in $\Sigma$.

\subsection{Review of Results} \label{section:csmToric}
In this section we review the results which will provide the basis for Algorithm \ref{algorithm:toricAlg} below. The main ingredient in this algorithm is the following result of Barthel, Brasselet and Fieseler \cite{barthel1992classes}.

\begin{propn}[Main Theorem of Barthel, Brasselet and Fieseler \cite{barthel1992classes}] Let $X_{\Sigma}$ be an $n$-dimensional complex toric variety specified by a fan $\Sigma$. We have that the  Chern-Schwartz-MacPherson class of $X_{\Sigma}$ can be written in terms of orbit closures as \begin{equation}
c_{SM}(X_{\Sigma})=\sum_{\sigma \in \Sigma} [V(\sigma)] \;\;\; \in A^*(X_{\Sigma})_{\mathbb{Q}} \label{eq:csm_toric_compute}
\end{equation} where $V(\sigma) $ is the closure of the torus orbit corresponding to $\sigma$. 
\label{propn:csm_toric_compute}
\end{propn}

We now recall the definition of the multiplicity of a simplicial cone, for more details see \S{6.4} of Cox, Little, and Schenck \cite{david2011toric}. Let $N$ be an integer lattice with dual lattice $M$, let $\sigma=\left\langle v_1, \dots, v_d \right\rangle$ be a simplicial cone and let \begin{equation}
N_{\sigma}=\mathrm{Span}(\sigma) \cap N \label{eq:N_sigma},
\end{equation} recall that $\mathrm{Span}(\sigma) \subset N_{\RR}$ is the smallest subspace of the vector space $N_{\RR}$ which contains $\sigma$. We note that the index of the subgroup $\ZZ v_1 +\cdots +\ZZ v_d \subset N_{\sigma}$ in $N_{\sigma}$ is finite. We define the multiplicity of $\sigma$ as \begin{equation}
\mathrm{mult}(\sigma)=\left[ N_{\sigma}: \ZZ v_1 +\cdots +\ZZ v_d \subset N_{\sigma} \right]\label{eq:mult_def}
\end{equation} where $[G:H]$ denotes the index of a subgroup $H$ in a group $G$. In practice we shall employ Lemma \ref{lemma:MultCone} to compute $\mathrm{mult}(\sigma)$. Specifically Lemma \ref{lemma:MultCone} will allow us to compute the multiplicity of a simplicial cone. Since we only consider complete simplicial toric varieties in Algorithm \ref{algorithm:toricAlg} this lemma may be used to compute the multiplicity in all cases considered here. %Lemma \ref{lemma:MultCone} is a modified version of Proposition 11.1.8.\ of Cox, Little, Schenck \cite{david2011toric}, it will allow us to compute the multiplicity of a simplicial cone. We have slightly altered the statement of the result to explicitly show how we will compute these multiplicities in practice. 

To compute the classes $[V(\sigma)]$ appearing in \eqref{eq:csm_toric_compute} we will employ the following proposition combined with Proposition \ref{propn:ChowRingDef}.
\begin{propn}[Theorem 12.5.2. of Cox, Little, Schenck \cite{david2011toric}]
Assume that $X_{\Sigma}$ is complete and simplicial. If $\rho_1, \dots, \rho_d \in \Sigma(1)$ are distinct and if $\sigma=\rho_1+\cdots+\rho_d \in \Sigma$ then in $A^*(X_{\Sigma})$ we have the following: \begin{equation}
[V(\sigma)]=\mathrm{mult}(\sigma)[V(\rho_1)]\cdot [V(\rho_2)]\cdots [V(\rho_d)]. 
\end{equation} Here $\mathrm{mult}(\sigma)$ will be calculated using Lemma \ref{lemma:MultCone}. \label{propn:class_of_orbit_closure}
 \end{propn}

\section{Algorithm and Performance}\label{section:AlgNPreform}
In this section we describe the process by which we turn the Main Theorem of Barthel, Brasselet and Fieseler \cite{barthel1992classes} (Proposition \ref{propn:csm_toric_compute}) into a computational method to find $c_{SM}$ classes of complete simplicial toric varieties. 
\subsection{Computing Multiplicitiy}\label{subsection:Multiplcites}
One of the main computational steps in Algorithm \ref{algorithm:toricAlg} below, for singular cases, is the computation of the multiplicity of a cone $\sigma \in \Sigma$. In practice this computation will be accomplished using Lemma \ref{lemma:MultCone}. This lemma is a modified version of Proposition 11.1.8.\ of Cox, Little, Schenck \cite{david2011toric}. We have altered the statement of the result to explicitly show how we will compute these multiplicities in practice. The main point here is to show how the definition of the multiplicity of a cone given in \eqref{eq:mult_def} can be phrased in terms of straightforward linear algebra computations in the cases considered in this note.  
\begin{lemma}[Modified version of Proposition 11.1.8.\ of Cox, Little, Schenck \cite{david2011toric}]
Let $N=\ZZ^n$ be an integer lattice. For a simplicial cone $\sigma=\rho_1+\cdots+\rho_d \subset N$ let $\mathfrak{M}_{\sigma} $ be the matrix with columns specified by the generating vectors of the rays $\rho_1,\dots,\rho_d$ which define the cone $\sigma$; we have \begin{equation}
\mathrm{mult}(\sigma)= \left| \det( \mathrm{Herm}(\mathfrak{M}_{\sigma})) \right|
\end{equation} where $\mathrm{Herm}(\mathfrak{M}_{\sigma})$ denotes the Hermite normal form of matrix $\mathfrak{M}_{\sigma}$ with all zero rows and/or zero columns removed. Further $\mathrm{mult}(\sigma)=1$ if and only if $U_{\sigma}$ is smooth.% and if $\tau$ is a face of $\sigma$ $\mathrm{mult}(\tau) \leq\mathrm{mult}(\sigma).$
\label{lemma:MultCone}
\end{lemma}
\begin{proof}
Suppose $\rho_1= \left\langle u_1 \right\rangle, \dots, \rho_d= \left\langle u_d \right\rangle$ so that we can write $\sigma= \left\langle u_1, \dots, u_d \right\rangle$. In Proposition 11.1.8.\ of Cox, Little, Schenck \cite{david2011toric} it is shown that if $e_1, \dots, e_d$ is a basis for $N_{\sigma}$ (see \eqref{eq:N_sigma}) and $u_i=\sum_{j=1}^d a_{i,j}e_j=E [a_{i,j}]$ (where $E$ is the $n \times d$ matrix with columns $e_1, \dots, e_d$) then we have that \begin{equation}
\mathrm{mult}(\sigma)= \left| \det \left( [a_{i,j}] \right) \right|. \label{eq:det_mult_CoxLittleSchenck}
\end{equation}

The matrix $\mathfrak{M}_{\sigma}$ defined by the rays $\rho_1,\dots,\rho_d$ is the $n\times d$ matrix with columns given by the vectors $u_1,\dots, u_d$. Note that $\mathfrak{M}_{\sigma}$ has rank $d$. Choose $e_1,\dots, e_d$ to be a basis of $N_{\sigma}$ so that the matrix $E$ with columns $e_1,\dots, e_d$ has the form $$E=\left[ \begin{array}{c}
\tilde{E}\\ \mathbf{0} \end{array} \right]$$ with $\det(\tilde{E})=1$. Now since $\mathfrak{M}_{\sigma}$ has rank $d$ we may write $$\mathfrak{M}_{\sigma}=\left[ \begin{array}{c}
\mathrm{Herm}(\mathfrak{M}_{\sigma})\\ \mathbf{0} \end{array} \right]T$$ for $\mathrm{Herm}(\mathfrak{M}_{\sigma})$ the $d\times d$ matrix obtained from the Hermite normal form of $ \mathfrak{M}_{\sigma}$ with the zero rows removed and $T$ a $d\times d$ unimodular matrix. Then we have that  $$
\left[ \begin{array}{c}
\tilde{E}\\ \mathbf{0} \end{array} \right] [a_{i,j}]=\left[ \begin{array}{c}
\mathrm{Herm}(\mathfrak{M}_{\sigma})\\ \mathbf{0} \end{array} \right] T,
$$ and hence $\tilde{E}[a_{i,j}]=\mathrm{Herm}(\mathfrak{M}_{\sigma}) T$. Note that $\det(\tilde{E})=\det(T)=1$, this gives that $\det([a_{i,j}])=\det \left(\mathrm{Herm}(\mathfrak{M}_{\sigma}) \right)$ as claimed.

 The Hermite normal form of $\mathfrak{M}_{\sigma}$ is obtained by performing unimodular column operations on $\mathfrak{M}_{\sigma}$ and thus represents a change of basis of $N_{\sigma}$, we may call this new basis for $N_{\sigma} $ $e_1, \dots, e_d$. Since $\mathfrak{M}_{\sigma}$ has rank $d$ then removing the zero rows we obtain the matrix $ \mathrm{Herm}(\mathfrak{M}_{\sigma})$ and we may then take this matrix to be the matrix $[a_{i,j}]$ in \eqref{eq:det_mult_CoxLittleSchenck} since the matrix $ \mathrm{Herm}(\mathfrak{M}_{\sigma})$ specifies the change of basis for $N_{\sigma}$ from $u_1,\dots, u_d$ to $e_1, \dots, e_d$.
The matrix $\mathfrak{M}_{\sigma}$ defined by the rays $\rho_1,\dots,\rho_d$ is the matrix with columns given by the vectors $u_1,\dots, u_d$, that is $\mathfrak{M}_{\sigma}=[u_1, \dots , u_d]$, further $\mathfrak{M}_{\sigma}$ has rank $d$.  $N_{\sigma}=\mathrm{Span}(\sigma) \cap N$ is the lattice generated by the columns of the matrix $\mathfrak{M}_{\sigma}$, that is $N_{\sigma}=\left\lbrace y \; | \; y= \mathfrak{M}_{\sigma}x , \; x \in \RR^d \right\rbrace \cap N$. From the definition of the Hermite form we have that $$\mathfrak{M}_{\sigma}T= \left[ \begin{array}{c}
\mathrm{Herm}(\mathfrak{M}_{\sigma}) \\ \mathbf{0} \end{array} \right] $$ for some unimodular matrix $T$. Thus we have $$N_{\sigma}=\left\lbrace y \; | \; y= \left[ \begin{array}{c}
\mathrm{Herm}(\mathfrak{M}_{\sigma}) \\ \mathbf{0} \end{array} \right] x , \; x \in \RR^d \right\rbrace\cap N ,$$ meaning we may take the matrix $[a_{i,j}]=\mathrm{Herm}(\mathfrak{M}_{\sigma})$ in \eqref{eq:det_mult_CoxLittleSchenck} and the conclusion follows. 

The remaining statements are given in the form stated above in Proposition 11.1.8. of Cox, Little, Schenck \cite{david2011toric}. 
\end{proof}
\subsection{Algorithm}\label{subsection:Algorithm}
In Algorithm \ref{algorithm:toricAlg} we present an algorithm to compute $c_{SM}(X_{\Sigma})$ for a complete, simplicial toric variety $ X_{\Sigma}$ defined by a fan $\Sigma$. Note that we represent $[V({\rho_j})]$ as $x_j$ via the isomorphism given in Proposition \ref{propn:ChowRingDef}. 

%Let $\omega_1,\dots,\omega_t$ be a basis for $A_0(X_{\Sigma})$ (this can be computed either using Proposition \ref{propn:ChowGroupsBasisToric} or by finding a monomial basis of the quotient ring presentation of Proposition \ref{propn:ChowRingDef} using standard methods). Then we may write the dimension zero component of $c_{SM}(X_{\Sigma})$ as $$
%\left(c_{SM}(X_{\Sigma})\right)_0=b_1\omega_1+\dots +b_1\omega_t,
%$$ hence we have that the Euler characteristic of $X_{\Sigma}$ is given by \begin{equation}
%\chi(X_{\Sigma})=\int c_{SM}(X_{\Sigma}) =b_1+\dots +b_t.\label{eq:Euler_from_csm_toric}
%\end{equation}

\begin{algorithm}

%\begin{itemize}

 \textbf{Input:} A complete, simplicial toric variety $ X_{\Sigma}$ defined by a fan $\Sigma$ with $\Sigma(1)=\left\lbrace \rho_1, \dots , \rho_r \right\rbrace$ and a boolean, Euler\_only, indicating if only the Euler characteristic is desired. We assume $\dim(X_{\Sigma}) \geq 1$.\newline 
 \textbf{Output:} $c_{SM}(X_{\Sigma})$ in $A^*(X_{\Sigma})_{\mathbb{Q}}\cong  \mathbb{Q}[x_1,\dots,x_r]/(\mathcal{I}+\mathcal{J})  $ and/or the Euler characteristic $\chi(X_{\Sigma})$, if Euler\_only=true then only $\chi(X_{\Sigma})$ will be computed. 
\begin{itemize}
\item Compute the rational Chow ring $A^*(X_{\Sigma})_{\mathbb{Q}}\cong  \mathbb{Q}[x_1,\dots,x_r]/(\mathcal{I}+\mathcal{J})  $ using Proposition \ref{propn:ChowRingDef}.
\item $\mathrm{csm}=0$.
\item \textbf{For $i$ from $\dim(X_{\Sigma})$ to $1 $:}
\begin{itemize}
\item $\mathrm{orbits}=$ all subsets of $\Sigma(1)=\left\lbrace \rho_1, \dots , \rho_r \right\rbrace$ containing $i$ elements.
\item $\mathrm{total}=0$.
\item \textbf{For $ \rho_{j_1},\dots, \rho_{j_s}$ in $ \mathrm{orbits}$:}
\begin{itemize}
\item $\sigma=\rho_{j_1}+\cdots+ \rho_{j_s}$.
\item Find $w=\mathrm{mult}(\sigma)$ using Lemma \ref{lemma:MultCone}.
\item $[V(\sigma)]=\mathrm{mult}(\sigma)[V({\rho_{i_1}})]\cdots [V({\rho_{i_s}})]=w \cdot x_{i_1} \cdots x_{i_s}$.
\item $\mathrm{total}=\mathrm{total}+[V(\sigma)]$.
\end{itemize}
\item $\mathrm{csm}=\mathrm{csm}+\mathrm{total}$.
\item \textbf{If $i==\dim(X_{\Sigma})$:} \begin{itemize}

\item Set $(c_{SM}(X_{\Sigma}))_0=\mathrm{csm}$.
\item Set $\chi(X_{\Sigma})=$ sum of the coefficients of the monomials in $(c_{SM}(X_{\Sigma}))_0$.

\item \textbf{If Euler\_only==true:}
 \begin{itemize}
 \item \textbf{Return $\chi(X_{\Sigma})$}.
\end{itemize}
\end{itemize}
%\item $$\mathrm{csm}=\mathrm{csm}+\sum_{\left\lbrace \rho_{j_1},\dots, \rho_{j_s}\right\rbrace \in \mathrm{orbits}} $$
\end{itemize}
\item Set $c_{SM}(X_{\Sigma})=\mathrm{csm}$.
%\item Let $\omega_1,\dots,\omega_t$ be a basis for $A_0(X_{\Sigma})$
%\item Set $\chi(X_{\Sigma})=\int c_{SM}(X_{\Sigma})$.
\item \textbf{Return $c_{SM}(X_{\Sigma})$ and/or $\chi(X_{\Sigma})$ }.
\end{itemize}
%\end{itemize}
\label{algorithm:toricAlg}
\end{algorithm} 

 We note that Algorithm \ref{algorithm:toricAlg} is strictly combinatorial; hence the runtime depends only on the combinatorics of the fan $\Sigma$ defining the toric variety. %As such the run times of Algorithm \ref{algorithm:toricAlg} should be quite reasonable even on large examples. 
 
%We now give an example of the output and input to Algorithm \ref{algorithm:toricAlg} and Algorithm \ref{algorithm:toricAlgEuler} for a singular complete simplicial toric variety.
 
% \subsection{Performance}\label{section:PerfomancecsmToric}
  %Since the goal here is to see how long the algorithm takes on larger examples we will primarily use complete simplicial toric varieties $X_{\Sigma}$ which can be constructed using built-in methods in the ``NormalToricVarieties" Macaulay2 \cite{M2} package. Because of this nearly all the examples considered will be smooth (meaning we are just computing Chern classes). 
 
 In this subsection we give the run times for Algorithm \ref{algorithm:toricAlg} applied to a variety of examples. Consider a complete simplicial toric variety $X_{\Sigma}$. We give two alternate implementations of Algorithm \ref{algorithm:toricAlg} to reflect what we can expect the timings to be in both the smooth cases and singular cases. 
 
Specifically the running times in Table \ref{table:CSMToric} for Algorithm \ref{algorithm:toricAlg} marked with a $\dagger$ check the input to see if the given fan $\Sigma$ defines a smooth toric variety, if it does these implementations use the fact that $\mathrm{mult}(\sigma)=1$ for all $\sigma \in \Sigma$ and hence do not compute the Hermite normal forms and their determinates in Lemma \ref{lemma:MultCone}. However to show how the algorithm would perform on a singular input of a similar size and complexity we also give running times for an implementation which always computes the Hermite forms and their determinates in Lemma \ref{lemma:MultCone}. 

In this way we see in a precise manner what the extra cost associated to computing the $c_{SM}$ class and Euler characteristic of a singular toric variety would be in comparison to the cost of computing a smooth toric variety defined by a fan having similar combinatorial structure.  Hence the running time for a given example would be very similar to that of a singular toric variety with a similar number and dimension of cones to those considered in the examples in Table \ref{table:CSMToric}. 

By default the implementation of Algorithm \ref{algorithm:toricAlg} in our ``CharToric" package checks if the input defines a smooth toric variety, i.e.\ performs the procedure of the implementations marked with $\dagger$. As such the performance of the package methods on smooth cases can be expected to be that of Algorithm \ref{algorithm:toricAlg} $\dagger$ in Table \ref{table:CSMToric} below. 

We also remark that the extra cost in the singular case (or in the case where we don't check the input) comes entirely from performing linear algebra with integer matrices. As such the running times in these cases could perhaps be somewhat reduced by using a specialized integer linear algebra package. To give a rough quantification of what performance improvement one might expect from this we performed some testing using LinBox \cite{Linbox} and PARI \cite{PARI2} via Sage \cite{sage} on linear systems of similar size and structure to those arising in the examples in Table \ref{table:CSMToric}. In this testing we found that the specialized algorithms seemed to be around two to three times faster than the linear algebra methods used by our implementation in the ``CharToric" package, however this testing is by no means conclusive. 

In any case it seems reasonable to conclude that some performance increase could be expected, for singular examples, if one used a specialized, fast integer linear algebra package to compute the Hermite forms and determinates arising in Algorithms \ref{algorithm:toricAlg}. Finally we note that additional efficiencies in implementation might also be found by a more careful implementation of the combinatorial procedures in a compiled language such as C or C++ rather than the Macaulay2 \cite{M2} language used here, which is an interpreted language. 
 \begin{table}[h!]
\resizebox{\linewidth}{!}{

\begin{tabular}{@{} l *8c @{}}
\toprule 
 \multicolumn{1}{c}{{\color{Ftitle} \textbf{Input}}} & {\color{Ftitle} Alg.\ \ref{algorithm:toricAlg} $\dagger$ }  & {\color{Ftitle} Alg.\ \ref{algorithm:toricAlg} (Euler only) $\dagger$ }  & {\color{Ftitle} Alg. \ref{algorithm:toricAlg} }  & {\color{Ftitle} Alg.\ \ref{algorithm:toricAlg} (Euler only)}  &\color{Ftitle} Chow Ring (Prop.\ \ref{propn:ChowRingDef}) \\ 
 \midrule 
% Example \ref{example:toric_Csm}& {\color{line} 0.0s}&{\color{line} 0.0s}& 0.0s&0.0s & 0.1 s\\
  $\pp^6$  & {\color{line} 0.0s}&{\color{line} 0.0s}& 0.0s&0.0s& 0.1 s\\
 %  $ \pp^{12}$  & {\color{line} 0.2s}&{\color{line} 0.0s}& 3.8s&0.0s& 0.3 s\\
    $ \pp^{16}$  & {\color{line} 5.3s}&{\color{line} 0.0s}& 85.4s&0.0s& 0.7 s\\
   $\pp^5 \times \pp^6$  & {\color{line} 0.3s}&{\color{line} 0.0s}& 3.7s&0.0s& 1.2 s\\
  $\pp^5 \times \pp^8$  & {\color{line} 1.1s}&{\color{line} 0.0s}& 16.8s &0.1s& 2.1 s\\
   $\pp^8 \times \pp^8$  & {\color{line} 12.0s}&{\color{line} 0.1s}& 168.5s&0.1s& 4.5 s\\
   $\pp^5 \times \pp^5 \times \pp^5$   & {\color{line} 12.8s}&{\color{line} 0.2s}& 156.7s&0.6s& 11.8 s\\
   $\pp^5 \times \pp^5 \times \pp^6$   & {\color{line} 28.4s}&{\color{line} 0.3s}& 387.1s&0.8s& 17.0 s\\
   Fano sixfold 123&  {\color{line} 0.3s}&{\color{line} 0.0s}& 1.0s&0.4s& 1.1 s\\
%  Fano sixfold 423&  {\color{line} 0.3s}&{\color{line} 0.0s}& 0.7s&0.2s& 0.8 s\\
  Fano sixfold 1007&  {\color{line} 0.4s}&{\color{line} 0.1s}& 1.0s&0.1s& 1.8 s\\
\bottomrule
 \end{tabular}}
 \caption{ In the table we present the time to compute the Chow ring separately from the time required for the other computations, as such the total run time for each algorithm will be the time listed in its column plus the time to compute the Chow ring if the Chow ring is not already known. Computations were performed using Macaulay2 \cite{M2} on a computer with a 2.9GHz Intel Core i7-3520M CPU and 8 GB of RAM. The Fano sixfolds are those built by the smoothFanoToricVariety method in the ``NormalToricVarieties" Macaulay2 \cite{M2} package. $\pp^n$ denotes a projective space of dimension $n$.}
% \caption[The $c_{SM}$ class and Euler chacteristic of compete simplicial toric varieties]{Running times for Algorithm \ref{algorithm:toricAlg} ($c_{SM}$ and Euler) and Algorithm \ref{algorithm:toricAlgEuler} (only Euler), our  algorithms to compute the $c_{SM}$ class and Euler chactereistic of a compete simplicial toric variety. The $\dagger$ denotes that for these versions of the algorithms the input is checked for smoothness. Meaning that, in the $\dagger$ versions, if the input is found to be smooth we know $\mathrm{mult}(\sigma)=1$ for all cones $\sigma \in \Sigma$ and hence we do not compute Hermite normal forms and determinates, see Lemma \ref{lemma:MultCone}. In this table we present the time to compute the Chow ring seperately from the time reqired for the other computations, as such the total run time for each algorithm will be the time listed in its column plus the time to compute the Chow ring if the Chow ring is not already known. Computations were performed using Macaulay2 \cite{M2} on a computer with a 2.9GHz Intel Core i7-3520M CPU and 8 GB of RAM.}
 \label{table:CSMToric}
 \end{table}

%%%%%%%%%%%%%%%%%%%%%%%%%%%%%%%%%%%%%%%

%%%%%%%%%%%%%%%%%%%%%%%%%%%%%%%%%%%%%%%%
%% BIBLIOGRAPHY
%%%%%%%%%%%%%%%%%%%%%%%%%%%%%%%%%%%%%%%%
\newpage
\bibliographystyle{plain}
\bibliography{refs}
\end{document}